 \documentclass[draft]{article}

\usepackage{amsmath,amsfonts,amsthm,amssymb,amscd,cancel,color}
\usepackage{enumitem}
\usepackage{verbatim}
\usepackage[dvipdfmx]{graphicx}
\usepackage{ulem,color}
\usepackage{mathrsfs}

\renewcommand{\Re}{\mathop{\rm Re}\nolimits}
\renewcommand{\Im}{\mathop{\rm Im}\nolimits}

\theoremstyle{plain}
\newtheorem{theorem}{Theorem}[section]
\newtheorem{lemma}[theorem]{Lemma}
\newtheorem{proposition}[theorem]{Proposition}

\theoremstyle{definition}

\theoremstyle{remark}
\newtheorem{remark}[theorem]{Remark}

\newtheorem{claim}[theorem]{Claim}

\newcommand{\R}{{\mathbb R}}

\newcommand{\Z}{{\mathbb Z}}

\newcommand{\N}{{\mathbb N}}

\def\im{{\rm i}}

\newcommand{\C}{\mathbb{C}}

\def\({\left(}
\def\){\right)}
\def\<{\left\langle}
\def\>{\right\rangle}

\numberwithin{equation}{section}

\begin{document}

\title{Asymptotic stability of soliton for discrete nonlinear Schrödinger equation on one-dimensional lattice
}

\author{Masaya Maeda and Masafumi Yoneda}
\maketitle

\begin{abstract}
In this paper we give a simple and short proof of asymptotic stability of soliton for discrete nonlinear Schrödinger equation near anti-continuous limit.
Our novel insight is that the analysis of linearized operator, usually non-symmetric, can be reduced to a study of simple self-adjoint operator almost like the free discrete Laplacian restricted on odd functions.
\end{abstract}

\section{Introduction}

In this paper, we study the discrete nonlinear Schr\"odinger equation (DNLS) on $\Z$:
\begin{align}\label{dnls}
\im \partial_t u = -\Delta_{\mathrm{d}}u - |u|^6 u,\ u:\R\times \Z\to \C,
\end{align}
where $\Delta_{\mathrm{d}}$ is the discrete Laplacian given by $\Delta_{\mathrm{d}}f(x)=f(x+1)-2f(x)+f(x-1)$. 
\begin{remark}
We have chose to work on the specific nonlinearity $-|u|^6u$.
However, it will be clear that the proof and the result of this paper also holds for general nonlinearity $g(|u|^2)u$ with smooth $g$ satisfying $g(0)=g'(0)=g''(0)=0$.
\end{remark}

DNLS \eqref{dnls} appear in many models in physics such as Bose-Einstein condensation in optical lattice \cite{Cataliotti01Science} and photonic lattice \cite{Efremidis02PRE}.
The aim of this paper is to study the stability property of bound states (soliton) solutions $e^{\im \omega t}\phi_\omega$. 
To state our result precisely, we set,
\begin{align*}
\|u\|_{l^p_a}:=\|e^{a|\cdot|}u\|_{l^p},\quad
(u,v):=\sum_{x\in \Z}u(x)\overline{v(x)}\quad \text{and}\quad  \<u,v\>:=\Re(u,v).
\end{align*}

We will consider large and concentrated solitons given by the following Proposition:
\begin{proposition}\label{prop:nbs}
There exist $\omega_0>0$ and $C>0$ s.t.\ there exist $\phi_{\cdot} \in C^\omega((\omega_0,\infty),l^2_{10})$ s.t. 
\begin{align}\label{sp}
0=-\Delta_{\mathrm{d}} \phi_\omega +\omega \phi_\omega - |\phi_\omega|^6 \phi_\omega,
\end{align}
and
\begin{align}
\sum_{j=0}^2\omega^j 
\| \partial_\omega^j\phi_\omega - \partial_\omega^j\(\omega^{\frac{1}{6}}\delta_0\) \|_{l^{2}_{10}}&\leq C\omega^{-\frac{5}{6}},\label{sp:est1}\\
\sum_{j=0}^{2}\omega^j\|P_0^\perp \partial_{\omega}^j \phi_{\omega}\|_{l^2_{10}}&\leq C \omega^{-\frac{5}{6}},\label{sp:est2}
\end{align}
where $P_0^\perp = 1 - (\cdot,\delta_0)\delta_0$ and $\delta_0(x)=1$ if $x=0$ and $\delta_0(x)=0$ if $x\neq 0$.
\end{proposition}

The main result of this paper is the asymptotic stability result for the soliton $e^{\im \omega t }\phi_{\omega}$ given in Proposition \ref{prop:nbs} (notice that if $\phi_{\omega}$ satisfies \eqref{sp}, then $e^{\im \omega t}\phi_{\omega}$ is a solution of \eqref{dnls}) for $\omega$ sufficiently large.
In particular, we prove the following:

\begin{theorem}\label{thm:main}
There exists $\omega_1\geq \omega_0$, where $\omega_0$ is given in Proposition \ref{prop:nbs}, s.t.\ for $\omega_*>\omega_1$, there exist $\delta_0>0$ and $C>0$ s.t.\ if $\epsilon:=\|u-\phi_{\omega_*}\|_{l^2}<\delta_0$, then there exist $\theta,\omega\in C^\infty([0,\infty),\R)$, $\omega_+>\omega_1$ and $\xi_+\in l^2$ s.t.
\begin{align}
&\lim_{t\to \infty} \| u(t)-e^{\im \theta(t)}\phi_{\omega(t)}-e^{\im t \Delta_{\mathrm{d}}}\xi_+\|_{l^2}=0,\label{eq:main1}\\
&\lim_{t\to \infty}\omega(t)=\omega_+ \ \text{and} \label{eq:main2}\ 
|\log \omega_*- \log \omega_+|+\|\xi_+\|_{l^2}\leq C\epsilon.
\end{align}
\end{theorem}

Considering large soliton concentrated on a finite set is equivalent to studying the bound states near the so-called ``anti-continuous limit".
Indeed, setting $u(t,x)=\omega^{\frac{1}{6}}v(\omega t,x)$, $v$ satisfies
\begin{align}\label{dNLSrescale}
\im \partial_t v = - \epsilon \Delta_{\mathrm{d}} v - |v|^6v,
\end{align}
where $\epsilon = \omega^{-1}$.
By such rescaling, bound states given in Proposition \ref{prop:nbs} is rescaled to $e^{\im t} \psi_{\epsilon}(x)$ with $\psi_{\epsilon}\sim \delta_0$.
In the anti-continuous limit $\epsilon\to 0$,  \eqref{dNLSrescale} reduces to an infinite system of unrelated ordinarily differential equations and in particular posses a solution $e^{\im t}\delta_0$.
The first rigorous treatment for the existence of solitons (also called discrete breathers in the context of discrete nonlinear Klein-Gordon equations) branching from the above solution, was given by MacKay and Aubry \cite{MA94N} followed by \cite{Aubry97PD,JA97N}.

The asymptotic stability problem of soliton near the anti-continuous limit is no less important than the problem of existence, but it has not yet been thoroughly investigated.
The only asymptotic stability result we are aware is by Bambusi \cite{Bambusi13CMPdb} who studies the asymptotic stability of nonlinear discrete Klein-Gordon equations (for  linear/spectral stability see  \cite{PKF05PD,PS11PD}). 
Our result, Theorem \ref{thm:main}, is a DNLS version of Bambusi's result, but with a much simple proof, as explained below.
In this paper, we have chose to work on the rescaled setting to reduce the parameters.
Indeed, even if consider \eqref{dNLSrescale} we still need to consider a family of solitons $\psi_{\epsilon,\omega}$ with $\omega$ near $1$.
For asymptotic stability of bound states bifurcating from linear potential, see \cite{KPS09SIAM, CT09SIMA, MP12DCDS, Maeda17SIMA,Maeda21AA}.

We now explain the outline of the proof of Theorem \ref{thm:main}, which is short and simple and in particular avoids normal form transformation used in \cite{Bambusi13CMPdb}.
We start from a standard strategy initiated by \cite{BP92StP} for the asymptotic stability of solitons of nonlinear Schr\"odinger equations (see also \cite{SW90CMP} for small solitons).
That is, we decompose the solution $u$ near $\{e^{\im \theta}\phi_{\omega_*}\ |\ \theta\in \R\}$ as $u=e^{\im \theta}\phi_{\omega}+\xi$, with $i \xi$ orthogonal to $\{\im e^{\im \theta}\phi_{\omega}, e^{\im \theta}\partial_{\omega}\phi_{\omega}\}$.
Then, the problem is to study the dynamics of $\theta,\omega$ and $\xi$.
Roughly, the equation of $\xi$, which we obtain by substituting the ansatz into DNLS \eqref{dnls}, will have the form
\begin{align*}
\im \partial_t \xi = H_{\theta,\omega}\xi + O(\xi^2),
\end{align*}
where
\begin{align}
H_{\theta,\omega}\xi := \(-\Delta_{\mathrm{d}}+\omega - 4\phi_{\omega}^6\)\xi -3e^{2\im \theta}\phi_{\omega}^6  \overline{\xi}.
\end{align}
The "linear operator" $H_{\theta,\omega}$ is not $\C$-linear due to the complex conjugate.
Thus, it is natural to study the corresponding matrix $\C$-linear operator
\begin{align*}
\mathcal{H}_{\theta,\omega}:=\begin{pmatrix}
-\Delta_{\mathrm{d}} + \omega - 4\phi_{\omega}^6 & -3e^{2\im \theta}\phi_{\omega}^6 \\
3e^{-2\im \theta}\phi_{\omega}^6 & \Delta_{\mathrm{d}} - \omega + 4\phi_{\omega}^6
\end{pmatrix}.
\end{align*}
However, in general it is hard to study the spectral properties of the operator which are needed for the proof of asymptotic stability and one is forced to assume, for example, the nonexistence of embedded eigenvalues, edge resonances and internal modes (and if one admits internal modes, then one needs to assume Fermi Golden Rule property), see e.g. \cite{CM21DCDS,CM2111.02681}.

For our problem, the first attempt is to use the fact $\phi_{\omega}=\omega^{\frac{1}{6}}\(\delta_0+O(\omega^{-1})\)$ where the remainder is small and decaying exponentially, as proved in Proposition \ref{prop:nbs}.
So, we replace $\phi_{\omega}$ by $\omega^{\frac{1}{6}}\delta_0$ and consider
\begin{align}
\widetilde{\mathcal{H}}_{\theta,\omega}:=\begin{pmatrix}
-\Delta_{\mathrm{d}} + \omega - 4\omega \delta_0 & -3e^{2\im \theta}\omega \delta_0 \\
3e^{-2\im \theta}\omega \delta_0 & \Delta_{\mathrm{d}} - \omega + 4\omega \delta_0
\end{pmatrix}.
\end{align}
It is possible to study the spectral properties and show decay estimates related to the operator $\widetilde{\mathcal{H}}_{\theta,\omega}$.
However, one can reduce this operator one step further.
Now, recall that we were assuming $\im \xi \in  \{\im e^{\im \theta}\phi_{\omega}, e^{\im \theta}\partial_\omega \phi_{\omega}\}^{\perp}$ (here, we are considering real inner product).
Applying the approximation $\phi_{\omega}\sim \omega^{\frac{1}{6}}\delta_0$ to the \textit{orthogonality condition}, it will reduced to $\im \xi \in \{\im e^{\im \theta}\omega \delta_0, \frac{1}{6}e^{\im \theta}\omega^{-\frac{5}{6}}\}^\perp =\{\delta_0,\im \delta_0\}^\perp$, which simply means $\xi(0)=0$.
For $\xi$ satisfying $\xi(0)=0$, we have
\begin{align}
\widetilde{\mathcal{H}}_{\theta,\omega}\begin{pmatrix}
\xi\\ \overline{\xi}
\end{pmatrix}
=
\begin{pmatrix}
-\Delta_{\mathrm{d}} + \omega & 0 \\ 0 & \Delta_d -\omega
\end{pmatrix}
\begin{pmatrix}
\xi\\ \overline{\xi}
\end{pmatrix}.
\end{align}
Thus, $\widetilde{\mathcal{H}}_{\theta,\omega}$ reduces to diagonal matrix and there will be no point considering matrix operator anymore.
Thus, the task is now to study $\Delta_0 = P_0^\perp \Delta_{\mathrm{d}} P_0^\perp$ where $P_0^\perp$ is the projection on to the space $\{\xi\ |\ \xi(0)=0\}$, see Proposition \ref{prop:nbs}.

The analysis of $\Delta_0$ further reduces to the analysis of $\Delta_{\mathrm{d}}$ restricted on odd functions (see proof of Proposition \ref{prop:stz}, we also note that Bambusi \cite{Bambusi13CMPdb} studies similar operator).
Thus, we will obtain Strichartz estimate for free and the Kato smoothing estimates, which do not hold for $\Delta_{\mathrm{d}}$, from a simple fact that the edge resonance of $\Delta_{\mathrm{d}}$ is even.
Thus, we obtain the linear estimates needed (the idea using Kato smoothing is due to \cite{CT09SIMA}).
The rest of the paper is more of less standard except tracing the $\omega$ dependence of the error terms carefully.

From the above explanation, it should be clear that our strategy, reducing the linear operator to a simple operator $\Delta_0$ can be applied to other equations, nonlinearity and more over higher dimensions, where linear estimates are not well studied due to the complicated phase structure \cite{KKV08JFA,CI21N}.

This paper is organized as follows:
In section \ref{sec:lin}, we prove the linear estimates for $\Delta_0$.
Section \ref{sec:mod} will be devoted to the modulation argument and in particular we derive the equation of $\xi$ (see \eqref{nls:modcoor} and \eqref{eq:eta}) and $\theta,\omega$ (see, \eqref{eq:disceq}).
In section \ref{sec:prmain}, we prove Theorem \ref{thm:main} by bootstrapping argument (Proposition \ref{prop:main}).
Finally, in the appendix, we give an elementary proof of Proposition \ref{prop:nbs} which gives us the precise exponential decay rate of $\phi_{\omega}$ easily.

\section{Linear estimates}\label{sec:lin}
Recall that in Proposition \ref{prop:nbs}, $P_0^\perp$ was given by
\begin{align}\label{def:pperp}
P_0^\perp u(x)=\begin{cases}
u(x) & x\neq 0,\\
0 & x=0.
\end{cases}
\end{align}
We set the restriction of $\Delta_{\mathrm{d}}$ to $P_0^\perp l^2(\Z)=\{u\in l^2(\Z)\ |\ u(0)=0\}$ by $$\Delta_0:=P_0^\perp \Delta_{\mathrm{d}} :P_0^\perp l^2(\Z)\to P_0^\perp l^2(\Z).
$$
In particular, for $u\in P_0^\perp l^2(\Z)$, we have
\begin{align}
\(\Delta_0u\)(x)=\begin{cases}
u(\pm 2)-2u(\pm 1) & x=\pm 1,\\ u(x+1)-2u(x)+u(x-1) & |x|\geq 2
\end{cases}.
\end{align}
For an interval $I\subset\R$, we set
\begin{align}
\mathrm{Stz}(I):=L^\infty(I, l^2(\Z))\cap L^6(I,l^\infty(\Z)),\ \mathrm{Stz}^*:=L^1(I,l^2(\Z))+L^{\frac{6}{5}}(I,l^1(\Z)).
\end{align}

The linear estimates we use in this paper are the following:

\begin{proposition}[Strichartz and Kato smoothing estimates]\label{prop:stz}
Let $I\subset\R$ be an interval with $0\in I$.
Let $u_0:\Z\to\C$ and $f:\R\times \Z\to\C$.
Then, we have
\begin{align}
 \|e^{\im t \Delta_0}P^\perp_0 u_0\|_{\mathrm{Stz}(I)\cap L^2(I,l^{2}_{-1})}&\lesssim \|u_0\|_{l^2},\label{stz1}\\
\|\int_0^{\cdot} e^{\im (\cdot-s)\Delta_0}P_0^\perp f(s)\,ds\|_{\mathrm{Stz}(I)}&\lesssim \|f\|_{\mathrm{Stz}^*(I)+L^2(I,l^{2}_{1})},\label{stz2}\\
\|\int_0^{\cdot} e^{\im (\cdot-s)\Delta_0}P_0^\perp f(s)\,ds\|_{L^2(I,l^{2}_{-1})}&\lesssim \|f\|_{L^2(I,l^{2}_1)}.\label{kato}
\end{align}
\end{proposition}

\begin{proof}
First, let $P_{\pm}u(x)=u(x)$ if $\pm x\geq 1$ and $P_{\pm}u(x)=0$ if $\pm x\leq 0$. 
Then, we have $[\Delta_0,P_{\pm}]=0$ so we get
\begin{equation*}
e^{it\Delta_0}(P_++P_-)=(P_++P_-)e^{it\Delta_0}.
\end{equation*}
 It implies $\Delta_0=\Delta_+\oplus\Delta_-$ where $\Delta_{\pm}:=P_{\pm }\Delta_0:P_{\pm}l^2\to P_{\pm }l^2$.
Thus, it suffices to show each estimate for $\Delta_{\pm}$ and we will only consider $\Delta_+$.
Further, for  $u:\N\to \C$, we set $Tu:\Z\to \C$ by $Tu(x)=u(x) $ for $x\geq 1$, $Tu(0)=0$ and $Tu(x)=-u(-x)$ for $x\leq -1$.
Then, we have $T\Delta_+ = \Delta_{\mathrm{d}} T$. Therefore, it suffices to show the estimates \eqref{stz1}, \eqref{stz2} and \eqref{kato} for $\Delta_{\mathrm{d}}$ and $u_0$, $f$ restricted to odd functions.

Thus, we immediately have the Strichartz estimates by \cite{SK05N}:
\begin{align*}
\|e^{\im t \Delta_0}P_0^\perp u_0\|_{\mathrm{Stz}(I)}\lesssim \|u_0\|_{l^2}, \quad
\|\int_0^{\cdot} e^{\im (\cdot-s)\Delta_0}P_0^\perp f(s)\,ds\|_{\mathrm{Stz}(I)}\lesssim \|f\|_{\mathrm{Stz}^*(I)}.
\end{align*}
We next show the Kato smoothness estimate
\begin{align}\label{kato2}
\|e^{\im t \Delta_0}P_0^\perp u_0\|_{L^2(I,l^{2}_{-1})}&\lesssim \|u_0\|_{l^2},
\end{align}
which do not hold for $\Delta_{\mathrm{d}}$ with general $u_0$.
To show \eqref{kato2}, it suffices to show
\begin{align}\label{katosuff}
\sup_{\substack{\|u\|_{l^2_1}\leq 1\\ u:\mathrm{odd}}}\sup_{\Im \lambda\neq 0}\|(-\Delta_{\mathrm{d}}-\lambda)^{-1}u\|_{l^2_{-1}}\lesssim 1,
\end{align}
which is a sufficient condition for Kato smoothness, see \cite{Kato65MA,RS4}.
Using the fact that $u$ is odd, we have
\begin{align}\label{mirgreen}
\((-\Delta_{\mathrm{d}}-\lambda)^{-1}u\)(x)&=-\im\sum_{y\in\Z} \frac{e^{-\im \mu|x-y|}}{\sin \mu}u(y)=-\im\sum_{y>0} \frac{e^{-\im \mu|x-y|}}{\sin \mu}u(y)-\im\sum_{y<0} \frac{e^{-\im \mu|x-y|}}{\sin \mu}u(y)\\&
=-\im \sum_{y>0}\frac{e^{-\im \mu|x-y|}-e^{-\im \mu |x+y|}}{\sin \mu} u(y).\nonumber
\end{align}
Where $\cos \mu = 1+\frac{\lambda}{2}$ , $\Im \nu \leq 0$ and $\Re \mu \in[0 , 2\pi]$ . Thus, we see that the singularity is removed and it is easy to show \eqref{katosuff}.

The estimate
\begin{align}
\|\int_0^{\cdot} e^{\im (\cdot-s)\Delta_0}P_0^\perp f(s)\,ds\|_{\mathrm{Stz}(I)}&\lesssim \|f\|_{L^2(I,l^{2}_{1})},
\end{align}
follows from the dual of \eqref{kato2} and Christ-Kiselev lemma \cite{CK01JFA,SS00CPDE}.

Finally, we prove \eqref{kato} by a parallel argument of Lemma 8.7 of \cite{CM2109.08108}.
The following formula was proved in Lemma 4.5 of \cite{Mizumachi08JMKU}:
\begin{align}
&2\int_0^t e^{-\im (t-s)\Delta_0}P_0^\perp f(s)\,ds= \frac{\im }{\sqrt{2\pi}}\int_\R  e^{-\im t\lambda}(R(\lambda-\im 0)+R(\lambda+\im 0))P(\mathcal{F}^{-1}_{t} f)(\lambda)\,d\lambda \nonumber\\&\quad
 +\int_{0}^\infty e^{-\im (t-s)\Delta_0}P_0^\perp f(s)\,ds
 -\int_{-\infty}^0e^{-\im (t-s)\Delta_0}P_0^\perp f(s)\,ds
,\label{eq:Mform}
\end{align}
where, $R(\lambda)=(-\Delta_0-\lambda)^{-1}$ and $\mathcal{F}_t^{-1}$ is the inverse Fourier transform with respect to the $t$ variable.
For the 1st term of r.h.s.\ of \eqref{eq:Mform}, by Plancherel theorem, we have
\begin{align}
&\|\int_\R  e^{-\im t\lambda}(R(\lambda-\im 0)+R(\lambda+\im 0))P_0^\perp(\mathcal{F}^{-1}_{t} f)(\lambda)\,d\lambda\|_{L^2l^2_{-1}}\lesssim \max_{\pm} \|R(\lambda\pm \im 0)P_0^\perp(\mathcal{F}^{-1}_{t} f)(\lambda)\|_{L^2_{\lambda}l^2_{-1}}\nonumber\\&
\lesssim \max_{\pm} \sup_{\lambda \in \R}\|R(\lambda\pm \im 0)\|_{l^2_1\to l^2_{-1}}\|\mathcal{F}^{-1}_t f\|_{L^2_\lambda l^2_{1}}\lesssim \|f\|_{L^2 l^2_1},
\end{align}
where we have used \eqref{katosuff} in the 3rd inequality.
The 2nd and 3rd term can be estimated by using \eqref{kato2} and its dual.
Therefore, we have the conclusion.
\end{proof}

\section{Modulation argument}\label{sec:mod}
We set $\phi[\theta,\omega]:=e^{\im \theta}\phi_\omega$.
Then, by \eqref{sp}, we have
\begin{align}\label{sp1}
\im \omega \partial_\theta \phi[\theta,\omega] = -\Delta \phi[\theta,\omega]  - |\phi[\theta,\omega]|^6\phi[\theta,\omega].  
\end{align}
Further,  differentiating \eqref{sp1} w.r.t.\ $\theta$ and $\omega$, we obtain
\begin{align}
\mathcal{H}[\theta,\omega]\partial_\theta\phi[\theta,\omega]&=\im \omega \partial_\theta^2 \phi[\theta,\omega],\label{diffPhi1}\\
\mathcal{H}[\theta,\omega]\partial_\omega \phi[\theta,\omega]&=\im \partial_\theta \phi[\theta,\omega] + \im \omega \partial_\theta \partial_\omega \phi[\theta,\omega],\label{diffPhi2}
\end{align}
where
\begin{align}\label{def:linop}
\mathcal{H}[\theta,\omega]u:=-\Delta u +\mathcal{V}[\theta,\omega]u:=-\Delta_{\mathrm{d}} u - 4|\phi[\theta,\omega]|^6u -3|\phi[\theta,\omega]|^4\phi[\theta,\omega]^2\overline{u}.
\end{align}
\begin{remark}
The operator $\mathcal{H}[\theta,\omega]$ is not $\C$-linear but only $\R$-linear due to the complex conjugate in the last term of \eqref{def:linop}.
\end{remark}
It is easy to check that $\mathcal{H}[\theta,\omega]$ is symmetric w.r.t.\ the real innerproduct $\<\cdot,\cdot\>$.
That is, we have $\<\mathcal{H}[\theta,\omega]u,v\>=\<u,\mathcal{H}[\theta,\omega]v\>$.
We set
\begin{align*}
\Omega:=\<\im \cdot,\cdot\>.
\end{align*}
\begin{remark}
$\Omega$ is the symplectic form associated to discrete NLS \eqref{dnls}.
\end{remark}
We set
\begin{align*}
\mathbf{H}_{\mathrm{c}}[\theta,\omega]:=\{u\in  l^2 \ |\  \Omega(u,\partial_\theta \phi[\theta,\omega])=\Omega(u,\partial_\omega \phi[\theta,\omega])=0\}.
\end{align*}
and
\begin{align*}
\mathcal{T}_\omega(r):=\{u\in  l^2 \ |\ \inf_{\theta\in \R}\|u-\phi[\omega,\theta]\|_{l^2}<r\}.
\end{align*}

\begin{lemma}[Modulation]\label{lem:mod}
For $\omega_*>\omega_0$, where $\omega_0$ is given in Proposition \ref{prop:nbs}, there exist $\delta>0$ s.t.\ there exist $\theta \in C^\infty(\mathcal{T}_{\omega_*}(\delta),\R)$ and $\omega\in C^\infty(\mathcal{T}_{\omega_*}(\delta),\R)$ s.t.
\begin{align}\label{xi:orth}
 \xi( u):= u-\phi[\theta( u),\omega( u)]\in \mathbf{H}_{\mathrm{c}}[\theta( u),\omega( u)].
\end{align}
\end{lemma}

\begin{proof}
Set
\begin{align}\label{Fformod}
\mathcal{F}(\theta,\omega, u):=\begin{pmatrix}
\Omega( u-\phi[\theta,\omega],\partial_\theta\phi[\theta,\omega])\\ 
\Omega( u-\phi[\theta,\omega],\partial_\omega\phi[\theta,\omega])
\end{pmatrix}.
\end{align}
Then, we have $\mathcal{F}(\theta,\omega,\phi[\theta,\omega])=0$ and
\begin{align*}
D_{(\theta,\omega)}\mathcal{F}(\theta,\omega, u)=q'(\omega)\begin{pmatrix}
0 & -1 \\ 1 & 0
\end{pmatrix}
+
\begin{pmatrix}
\Omega( u-\phi[\theta,\omega],\partial_\theta^2 \phi[\theta,\omega]) & \Omega( u-\phi[\theta,\omega], \partial_\theta\partial_\omega \phi[\theta,\omega])\\
\Omega( u-\phi[\theta,\omega],\partial_\theta\partial_\omega \phi[\theta,\omega]) & \Omega( u-\phi[\theta,\omega], \partial_\omega^2 \phi[\theta,\omega])
\end{pmatrix},
\end{align*} 
where $q(\omega)=\frac{1}{2}\|\phi[\theta,\omega]\|_{l^2}^2$.
Since $\partial_\theta \phi[\theta,\omega]=\im \phi[\theta,\omega]$, we have
\begin{align*}
q'(\omega)=\<\partial_\omega\phi[\theta,\omega],\phi[\theta,\omega]\>=\Omega(\partial_\omega\phi[\theta,\omega],\im \phi[\theta,\omega])=\Omega(\partial_\omega\phi[\theta,\omega],\partial_\theta\phi[\theta,\omega]).
\end{align*}
Thus, from Proposition \ref{prop:nbs} we have 
\begin{align}\label{qdash}
q'(\omega)\sim \omega^{-\frac{2}{3}}>0.
\end{align}
Therefore, $D_{(\theta,\omega)}\mathcal{F}(\theta,\omega,\phi[\theta,\omega])$ is invertible. By implicit function theorem we have the conclusion.
\end{proof}

\begin{lemma}\label{lem:xiintiest}
There exists $\omega_1>\omega_0$ s.t.\ for $\omega_*>\omega_1$, if $u\in \mathcal{T}_{\omega_*}(\delta)$, we have
\begin{align}\label{xi:est}
\|\xi(u)\|_{l^2}\lesssim \inf_{\theta\in \R}\|u-\phi[\theta,\omega_*]\|_{l^2}.
\end{align}
Here, $\delta>0$ is the constant (depending on $\omega_*$) given in Lemma \ref{lem:mod}.
\end{lemma}

\begin{proof}
Fix $u\in \mathcal{T}_{\omega_*}(0,\delta)$ and set $\epsilon=\inf_{\theta\in \R}\|u-\phi[\theta,\omega_*]\|_{l^2}$.
We take $\theta_0$ to satisfy $\epsilon=\|u-\phi[\theta_0,\omega_*]\|_{l^2}$  and set $v:=u-\phi[\theta_0,\omega_*]$ and for $s\in[0,1]$,
\begin{align*}
u[s]&:=\phi[\theta_0,\omega_*]+sv,\ w[s]:=w(u[s]),\ \theta[s]:=\theta(u[s]),\\
A[s]&:=\|u[s]-\phi[\omega[s],\theta[s]]\|_{l^2}.
\end{align*}
Notice that we have $u[0]=\phi[\theta_0,\omega_*]$, $\theta[0]=\theta_0$, $\omega[0]=\omega_*$, $u[1]=u$ and \eqref{xi:est} is equivalent to $A[1]\lesssim \epsilon$.
So, to show \eqref{xi:est} we prove the following claim.
\begin{claim}\label{claim:xi}
there exists $\omega_1>\omega_0$ and $C_0>0$ s.t.\ if $\omega_*>\omega_1$ and if
\begin{align}
|\omega[s]-\omega_*|&\leq C_0\omega_*^{5/6}\epsilon,\label{bootxini1}\\
A[s]&\leq C_0\epsilon,\label{bootxiinit2}
\end{align}
we have \eqref{bootxini1} and \eqref{bootxiinit2} with $C_0$ replaced by $C_0/2$.
\end{claim}

\begin{proof}[Proof of claim \ref{claim:xi}]
We assume \eqref{bootxini1} and \eqref{bootxiinit2} for all $\tau\in [0,s]$ for some $s\in (0,1)$.
We take $C_0= \omega_1^{\frac{1}{12}}$.
In this proof, when we use $\lesssim$ or $\sim$, the implicit constant will not depend on $\omega_1$, $\omega_*$ nor $s$.

From \eqref{bootxini1}, we have 
\begin{align}\label{omegasequiv}
\omega[s]\sim \omega_*.
\end{align}
From the fundamental theorem of calculus, we have
\begin{align}\label{Aboot}
A[s]&\leq \epsilon + \|\phi[\theta_s,\omega_s]-\phi[\theta_0,\omega_*]\|_{l^2\nonumber}\\&
\leq \epsilon + \int_0^s \(\|\partial_{\theta} \phi[\theta_ \tau,\omega_ \tau\|_{l^2} |D_u \theta(u_ \tau)v|+\|\partial_{\omega} \phi[\theta_ \tau,\omega_ \tau]\|_{l^2} |D_u \omega(u_ \tau)v|\)\,d\tau.
\end{align}
Differentiating $\mathcal{F}(\theta(u_\tau),\omega(u_\tau),u_\tau)=0$ w.r.t.\ $\tau$, where $\mathcal{F}$ is the function given in \eqref{Fformod}, we have
\begin{align}\label{eq:DuthetaDuxi}
\begin{pmatrix}
D_u \theta(u_\tau)v\\
D_u \omega(u_\tau)v
\end{pmatrix}
=
\(D_{(\theta,\omega)}\mathcal{F}(\theta(u_\tau),\omega(u_\tau),u_\tau)\)^{-1}
\begin{pmatrix}
\Omega( v,\partial_\theta\phi[\theta_\tau,\omega_\tau])\\ 
\Omega( v,\partial_\omega\phi[\theta_\tau,\omega_\tau])
\end{pmatrix}.
\end{align}
The determinant of $D_{(\theta,\omega)}\mathcal{F}$ can be explicitly written as
\begin{align}
&\mathrm{det}\(D_{(\theta,\omega)}\mathcal{F}\)(\theta_\tau,\omega_\tau,u_\tau)=-(q'(\omega_\tau))^2\label{detF}\\&+\Omega( u_\tau-\phi[\theta_\tau,\omega_\tau],\partial_\theta^2 \phi[\theta_\tau,\omega_\tau])
\Omega( u_\tau-\phi[\theta_\tau,\omega]_\tau, \partial_\omega^2 \phi[\theta_\tau,\omega_\tau])
+\Omega( u_\tau-\phi[\theta_\tau,\omega_\tau],\partial_\theta\partial_\omega \phi[\theta_\tau,\omega_\tau])^2.\nonumber
\end{align}
By \eqref{qdash} and \eqref{omegasequiv}, we have $q'(\omega_\tau)^{2}\sim \omega_*^{-\frac{4}{3}}=\omega_*^{-\frac{16}{12}}$ and
\begin{align*}
|\mathrm{det}\(D_{\theta,\omega}\mathcal{F}\)(\theta_\tau,\omega_\tau,u_\tau)+(q'(\omega_\tau))^2|\lesssim A[\tau]\omega_\tau^{\frac{1}{3}-2}\lesssim  \epsilon \omega_*^{-\frac{19}{12}}.
\end{align*}
Thus, we have
\begin{align}\label{detF3}
\mathrm{det}\(D_{\theta,\omega}\mathcal{F}\)(\theta_\tau,\omega_\tau,u_\tau)\sim \omega_*^{-\frac{4}{3}}.
\end{align}
Computing the r.h.s.\ explicitly and using Proposition \ref{prop:nbs}, \eqref{omegasequiv} and \eqref{detF3}, we have
\begin{align}
|D_u\theta(u_\tau)v|+\omega_*^{-1}|D_u \omega(u_\tau)v|\lesssim \omega_*^{-\frac{1}{6}}\(1 + A(s)\omega_*^{-\frac{1}{6}}\)\epsilon.
\end{align}
Substituting this bound in \eqref{Aboot} and $\omega_s=\omega_*+\int_0^t D_u\omega(u_\tau)v\,d\tau$, we have
\begin{align}
A[s]\lesssim \(1 + A(s)\omega_*^{-\frac{1}{6}}\)\epsilon\lesssim (1+\omega_1^{-\frac{1}{12}}\epsilon)\epsilon\lesssim \epsilon,\\
|\omega[s]-\omega_*|\lesssim \omega_*^{\frac{5}{6}}\(1 + A(s)\omega_*^{-\frac{1}{6}}\)\epsilon\lesssim \omega_*^{\frac{5}{6}}(1+\omega_1^{-\frac{1}{12}}\epsilon)\epsilon\lesssim \omega_*^{\frac{5}{6}}\epsilon.
\end{align}
Therefore, we have the conclusion.
\end{proof}
By claim \ref{claim:xi} and continuity argument, we obtain \eqref{bootxini1} and \eqref{bootxiinit2} with $s=1$ and in particular \eqref{xi:est}.
\end{proof}

Recall $P_0^\perp$ given in \eqref{def:pperp}.
For large $\omega$, the two spaces $P_0^\perp l^2$ and $\mathbf{H}_{\mathrm{c}}[\theta,\omega]$ become similar.
\begin{lemma}\label{lem:Q}
$\left.P_0^\perp\right|_{\mathbf{H}_{\mathrm{c}[\theta,\omega]}}$ is invertible.
Moreover, 
\begin{align}\label{def:Q}
Q[\theta,\omega]:=\(\left.P_0^\perp\right|_{\mathbf{H}_{\mathrm{c}}[\theta,\omega]}\)^{-1}:P_0^\perp  l^2 \to \mathbf{H}_{\mathrm{c}}[\theta,\omega],
\end{align}
is given by
\begin{align}\label{def:Q}
Q[\theta,\omega]u = u + e^{\im \theta} \(-\phi_{\omega}(0)^{-1}\Omega(u,\partial_\theta \phi[\theta,\omega]) + \im \partial_{\omega}\phi_\omega(0)^{-1} \Omega(u,\partial_\omega \phi[\theta,\omega])\)\delta_0,
\end{align}
and for $u\in P_0^\perp l^2$, we have
\begin{align}\label{qest}
\|u - Q[\theta,\omega]u\|_{l^2_{1}}\lesssim \omega^{-1} \|u\|_{l^2_{-1}}.
\end{align}
\end{lemma}

\begin{proof}
Since $\(P_0^\perp u\)(x)=u(x)$ for $x\neq 0$, the only possible form of the inverse of $\left.P_0^\perp\right|_{\mathbf{H}_{\mathrm{c}[\theta,\omega]}}$ is
\begin{align}\label{Qans}
Q[\theta,\omega]u=u + e^{\im \theta}q(u)\delta_0,
\end{align}
with $q(u)\in \C$.
Substituting  \eqref{Qans} into 
$ \Omega(Q[\theta,\omega]u,\partial_X\phi[\theta,\omega])=0$ for $X=\theta,\omega$, we  have
\begin{align}
\Re q(u)=-\phi_{\omega}(0)^{-1}\Omega(u,\partial_\theta \phi[\theta,\omega]),\ \Im q(u)=\partial_{\omega}\phi_\omega(0)^{-1} \Omega(u,\partial_\omega \phi[\theta,\omega]).
\end{align}
Since $v=Q[\theta,\omega]P_0^\perp u$ is the unique element of $l^2$ satisfying $v(x)=u(x)$ for $x\neq 0$ and $v\in \mathbf{H}_{\mathrm{c}}[\theta,\omega]$, we see $v=u$ for $u\in \mathbf{H}_{\mathrm{c}[\theta,\omega]}$.
Finally, \eqref{qest} follows from Proposition \ref{prop:nbs} and \eqref{def:Q}.
\end{proof}

%

In the following, we write $\theta(t):=\theta(u(t))$, $\omega(t):=\omega(u(t))$ and  $\xi(t):=\xi(u(t))$ where $u(t)$ is the solution of \eqref{dnls}.
Substituting $ u=\phi[\theta,\omega]+ \xi$ into the equation, we have
\begin{align}\label{nls:modcoor}
\im  \dot{ \xi} + \im  \partial_\theta \phi [\theta,\omega](\dot{\theta}-\omega) + \im \partial_\omega\phi [\theta,\omega]\dot{\omega} = \mathcal{H}[\theta,\omega] \xi - f[\theta,\omega, \xi] -|\xi|^6\xi,
\end{align}
where 
\begin{align}\label{def:f}
f[\theta,\omega, \xi]&:=\sum_{0\leq a\leq 4, 0\leq b\leq 3, 2\leq a+b\leq 6} A_{a,b}\phi[\theta,\omega]^{4-a}\overline{\phi[\theta,\omega]}^{3-b}\xi^a\overline{\xi}^b,
\end{align}
for some $A_{a,b}\in \N$.

We set 
\begin{align}\label{def:eta}
\eta(t):=P_0^\perp \xi(t).
\end{align}
Notice that from Lemma \ref{lem:Q}, we have $\xi (t)= Q[\theta(t),\omega(t)]\eta(t)$. 
Applying $P_0^\perp$ to \eqref{nls:modcoor}, we have
\begin{align}\label{eq:eta}
\im \dot{\eta}=&-\Delta_0 \eta  -P_0^\perp \Delta_{\mathrm{d}}(1-Q[\theta,\omega])\eta+P_0^\perp \mathcal{V}[\theta,\omega]\xi - \im P_0^\perp \partial_\theta \phi [\theta,\omega](\dot{\theta}-\omega) - \im P_0^\perp\partial_\omega\phi [\theta,\omega]\dot{\omega}\nonumber \\&
- P_0^\perp f(\theta,\omega,\xi) -P_0^\perp \(|\xi|^6 \xi\).
\end{align}

We next seek for the equation for $\dot{\theta}-\omega$ and $\dot{\omega}$.
First, taking the innerproduct $\<\eqref{nls:modcoor}, \partial_\omega\phi[\theta,\omega]\>$, we have
\begin{align}\label{eq:dotthetaomega}
\Omega(\dot{ \xi},\partial_\omega \phi[\theta,\omega])-q'(\omega)(\dot{\theta}-\omega)=\<\mathcal{H}[\theta,\omega] \xi,\partial_\omega\phi[\theta,\omega]\>-\< f[\theta,\omega, \xi]+|\xi|^6\xi,\partial_\omega\phi[\theta,\omega]\>.
\end{align}
Now, since $\frac{d}{dt}\Omega({ \xi},\partial_\omega \phi[\theta,\omega])=0$,
\begin{align}\label{eq:dotthetaomega1}
\Omega(\dot{ \xi},\partial_\omega \phi[\theta,\omega])=-\Omega( \xi,\partial_\omega^2\phi[\theta,\omega])\dot{\omega}-\Omega( \xi,\partial_\theta\partial_\omega \phi[\theta,\omega])\dot{\theta}.
\end{align}
From \eqref{diffPhi2}, we have
\begin{align}
\< \mathcal{H}[\theta,\omega] \xi,\partial_\omega\phi[\theta,\omega]\>&=\<  \xi,  \im \partial_\theta \phi[\theta,\omega] + \im \omega \partial_\theta \partial_\omega \phi[\theta,\omega]\>\nonumber\\&= -\omega \Omega (\xi, \partial_\theta \partial_\omega \phi[\theta,\omega]).\label{eq:dotthetaomega2}
\end{align}
Combining \eqref{eq:dotthetaomega}, \eqref{eq:dotthetaomega1} and \eqref{eq:dotthetaomega2} we have
\begin{align}\label{eq:dottheta}
\(q'(\omega)+\Omega( \xi,\partial_\theta\partial_\omega \phi[\theta,\omega])\)\(\dot{\theta}-\omega\)
+\Omega( \xi,\partial_\omega^2\phi[\theta,\omega])\dot{\omega}
=
\< f[\theta,\omega, \xi]+|\xi|^6\xi,\partial_\omega\phi[\theta,\omega]\>.
\end{align}
Similarly, taking the innerproduct $\<\eqref{nls:modcoor},\partial_{\theta}\phi[\theta,\omega]\>$, we have
\begin{align}\label{eq:dotomega}
\(q'(\omega)-\Omega(\xi,\partial_{\theta}\partial_{\omega}\phi[\theta,\omega])\)\dot{\omega} -\Omega(\xi,\partial_\theta^2\phi[\theta,\omega])(\dot{\theta}-\omega)=-\<f[\theta,\omega,\xi]+|\xi|^6\xi,\partial_{\theta}\phi[\theta,\omega]\>.
\end{align}
Combining \eqref{eq:dottheta} and \eqref{eq:dotomega}, we have
\begin{align}\label{eq:disceq}
A[\theta,\omega,\eta]\begin{pmatrix}
\dot{\theta}-\omega\\ \omega^{-1}\dot{\omega}
\end{pmatrix}
=
\begin{pmatrix}
\< f[\theta,\omega, \xi]+|\xi|^6\xi,\partial_\omega\phi[\theta,\omega]\>\\
-\omega^{-1}\<f[\theta,\omega,\xi]+|\xi|^6\xi,\partial_{\theta}\phi[\theta,\omega]\>
\end{pmatrix},
\end{align}
where
\begin{align}\label{def:A}
A[\theta,\omega,\eta]:=\begin{pmatrix}
q'(\omega)+\Omega(  Q[\theta,\omega]\eta,\partial_\theta\partial_\omega \phi[\theta,\omega]) & \omega \Omega(  Q[\theta,\omega]\eta,\partial_\omega^2\phi[\theta,\omega])\\
-\omega^{-1}\Omega( Q[\theta,\omega]\eta,\partial_\theta^2\phi[\theta,\omega]) & q'(\omega)-\Omega( Q[\theta,\omega]\eta,\partial_{\theta}\partial_{\omega}\phi[\theta,\omega])
\end{pmatrix}.
\end{align}
Here, we have multiplied $\omega^{-1}$ to \eqref{eq:dotomega} to adjust the scale.


\section{Proof of main theorem}\label{sec:prmain}

We set $X_T:=\mathrm{Stz}(0,T)\cap L^2((0,T),l^2_{-1})$.

\begin{proposition}\label{prop:main}
There exists $\omega_1>\omega_0$ s.t.\ for $\omega_*>\omega_1$, there exist $\epsilon_0\in (0,1)$ and $C_0>1$ with $C_0\epsilon_0<1$ s.t.\ for $T>0$, if $\epsilon:=\inf_{\theta\in\R}\|u(0)-\phi[\omega_*,\theta]\|_{l^2}<\epsilon_0$ and 
\begin{align}
\| \xi\|_{\mathrm{Stz}\cap L^2l^{2,-s}(0,T)}\leq C_0 \epsilon,\label{bootass1}\\
\|\omega^{-1}\dot{\omega}\|_{L^1\cap L^\infty(0,T)}+\|\dot{\theta}-\omega\|_{L^1\cap L^\infty(0,T)}\leq C_0\epsilon,\label{bootass2}
\end{align}
then the above holds with $C_0$ replaced by $C_0/2$.
\end{proposition}

In the following, we assume \eqref{bootass1} and \eqref{bootass2}.
Further, when we use $\lesssim $ or $\sim$, the implicit constant will not depend on $C_0$, $\epsilon$, $\omega_*$ nor $\omega_1$.
Since $\sup_{t\in(0,T)]}|\omega(t)-\omega_*|\leq \|\dot{\omega}\|_{L^1(0,T)}\leq C_0\epsilon$, assuming $\omega_1>2$ if necessary, we have 
\begin{align}\label{omegaequiv}
\omega(t)\sim \omega_*\text{ for all }t\in (0,T).
\end{align}
Further, since we have set $\eta=P_0^\perp \xi$ in section \ref{sec:mod}, from \eqref{def:pperp}, we have
\begin{align}\label{assetaest}
\|\eta\|_{\mathrm{Stz}\cap L^2((0,T),l^2_{-1})}\leq C_0\epsilon.
\end{align}
We start with the estimate of $\eta$.

\begin{lemma}\label{lem:booteta}
Under the assumption of Proposition \ref{prop:main}, we have
\begin{align}\label{eq:booteta}
\|\eta\|_{X_T}\lesssim \|\eta(0)\|_{l^2}+ C_0\omega_1^{-\frac{5}{6}}\epsilon+(C_0\epsilon)^7.
\end{align}
\end{lemma}

\begin{proof}
From \eqref{eq:eta} and Proposition \ref{prop:stz}, we have
\begin{align}
\|\eta\|_{X_T}\lesssim &
 \|\eta(0)\|_{l^2}+\|P_0^\perp \Delta_{\mathrm{d}}(1-Q[\theta,\omega])\eta\|_{L^2((0,T),l^2_1)}
+\|P_0^\perp \mathcal{V}[\theta,\omega]Q[\theta,\omega]\eta\|_{L^2((0,T),l^2_1)}\nonumber\\&
+\| P_0^\perp \partial_\theta \phi [\theta,\omega](\dot{\theta}-\omega) \|_{L^2((0,T),l^2_1)}
+ \|P_0^\perp\partial_\omega\phi [\theta,\omega]\dot{\omega}\|_{L^2((0,T),l^2_1)}
\nonumber\\&
+\|P_0^\perp f(\theta,\omega,\xi) \|_{L^2((0,T),l^2_1)} 
+ \|P_0^\perp \(|\xi|^6 \xi\)\|_{L^1((0,T),l^2)}\nonumber\\&
+\|\int_0^{\cdot}e^{\im \Delta_0 (\cdot-s)}P_0^\perp \(|\xi|^6 \xi(s)\)\,ds\|_{L^2((0,T),l^2_{-1})}.\label{est:eta1}
\end{align}
By $\|P_0^{\perp}\|_{l^2_1\to l^2_1}\leq 1$, $\|\Delta_{\mathrm{d}}\|_{l^2_1\to l^2_1}\lesssim 1$ and Lemma \ref{lem:Q}, we have
\begin{align}\label{est:eta2}
\|P_0^\perp \Delta_{\mathrm{d}}(1-Q[\theta,\omega])\eta\|_{L^2((0,T),l^2_1)}\lesssim \omega_1^{-1}\|\eta\|_{X_T}\lesssim C_0\omega_1^{-\frac{5}{6}}\epsilon.
\end{align}
By  and \eqref{sp:est2} and \eqref{def:linop}, we have $\|P_0^\perp \mathcal{V}[\theta,\omega]Q[\theta,\omega]\|_{l^2_{-1}\to l^2_1}\lesssim \omega_1^{-5}$.
Thus,
\begin{align}\label{est:eta3}
\|P_0^\perp \mathcal{V}[\theta,\omega]Q[\theta,\omega]\eta\|_{L^2((0,T),l^2_1)}\lesssim \omega_1^{-5} \|\eta\|_{X_T}\lesssim C_0\omega_1^{-\frac{5}{6}}\epsilon.
\end{align}
For the terms in the 2nd line of \eqref{est:eta1}, by \eqref{sp:est2}, we have
\begin{align}\label{est:eta4}
\| P_0^\perp \partial_\theta \phi [\theta,\omega](\dot{\theta}-\omega) \|_{L^2((0,T),l^2_1)}
+& \|P_0^\perp\partial_\omega\phi [\theta,\omega]\dot{\omega}\|_{L^2((0,T),l^2_1)}
\nonumber\\&\lesssim \omega_1^{-\frac{5}{6}}\|\dot{\theta}-\omega\|_{L^2(0,T)}+\omega_1^{-\frac{5}{6}}\|\omega^{-1}\dot{\omega}\|_{L^2(0,T)}\lesssim C_0 \omega^{-\frac{5}{6}}\epsilon.
\end{align}
For the first term of the 3rd line of \eqref{est:eta1}, by \eqref{sp:est2} and \eqref{def:f},
\begin{align}
\|P_0^\perp f(\theta,\omega,\xi)\|_{L^2((0,T),l^{2_1})}\lesssim \sum_{j=2}^6 \omega_1^{-\frac{5}{6}(7-j)} \|\xi\|_{L^\infty((0,T),l^2)}^{j-1}\|\xi\|_{L^2((0,T),l^{2}_{-1})}
\lesssim C_0\omega_1^{-\frac{5}{6}}\epsilon.\label{est:eta5}
\end{align}
For the 2nd term in the 3rd line, since $\mathrm{Stz}\hookrightarrow L^7l^{14}$, we have
\begin{align}\label{est:6}
\|P_0^\perp \(|\xi|^6\xi\) \|_{L^1((0,T),l^2)}\leq  \|\xi\|_{X_T}^7\leq (C_0\epsilon)^7.
\end{align}
Finally, for the last term of \eqref{est:eta1}, by Proposition \ref{prop:stz}, we have
\begin{align*}
\|\int_0^{\cdot}e^{\im \Delta_0 (\cdot-s)}P_0^\perp \(|\xi|^6 \xi(s)\)\,ds\|_{L^2((0,T),l^2_{-1})}&\leq \int_0^T \|e^{\im \Delta_0 (\cdot-s)}P_0^\perp \(|\xi|^6 \xi(s)\)\|_{L^2((0,T),l^2_{-1})}\,ds\nonumber\\&
\lesssim \int_0^T \|P_0^\perp \(|\xi|^6 \xi(x)\)\|_{l^2}\,ds= \|P_0^\perp \(|\xi|^6\xi\) \|_{L^1((0,T),l^2)}.
\end{align*}
Thus, from \eqref{est:6}, we have
\begin{align}
\|\int_0^{\cdot}e^{\im \Delta_0 (\cdot-s)}P_0^\perp \(|\xi|^6 \xi(s)\)\,ds\|_{L^2((0,T),l^2_{-1})}\lesssim (C_0\epsilon)^7.\label{est:7}
\end{align}
Combining \eqref{est:eta1}--\eqref{est:7}, we have \eqref{eq:booteta}.
\end{proof}

\begin{lemma}\label{lem:disc}
Under the assumption of Proposition \ref{prop:main}, we have
We have
\begin{align}\label{est:disc}
|\dot{\theta}-\omega|+|\omega^{-1}\dot{\omega}|\lesssim \omega^{-\frac{1}{3}} \|\eta\|_{l^2_{-1}}^2+\|\eta\|_{l^2_{-1}}^7.
\end{align}
\end{lemma}

\begin{proof}
First, from Proposition \ref{prop:nbs} and Lemma \ref{lem:Q} we have
\begin{align*}
|\Omega(  Q[\theta,\omega]\eta,\partial_\theta\partial_\omega \phi[\theta,\omega])|&\leq |\Omega(  \eta,P_0^\perp\partial_\theta\partial_\omega \phi[\theta,\omega])|+|\Omega((1-Q[\theta,\omega])\eta,\partial_\theta\partial_\omega \phi[\theta,\omega])|\\&\lesssim \omega^{-\frac{11}{6}}\|\eta\|_{l^2_{-1}}.
\end{align*}
Similarly, we have
\begin{align*}
|\omega \Omega(  Q[\theta,\omega]\eta,\partial_\omega^2\phi[\theta,\omega])|+|\omega^{-1}\Omega( Q[\theta,\omega]\eta,\partial_\theta^2\phi[\theta,\omega])|\lesssim \omega^{-\frac{11}{6}}\|\eta\|_{l^2_{-1}}.
\end{align*}
By \eqref{qdash}, we see that if $\|\eta\|_{l^2}\lesssim 1$, $A[\theta,\omega]$, defined in \eqref{def:A}, is invertible and we have 
\begin{align}\label{invAbound}
\|A[\theta,\omega,\eta]^{-1}\|_{\C^2\to \C^2}\lesssim \omega^{\frac{2}{3}}.
\end{align}
Next, 
\begin{align}
|\<f[\theta,\omega,\xi],\partial_\omega \phi[\theta,\omega]\>|&\lesssim \sum_{j=2}^6 \<|\phi[\theta,\omega]|^{7-j}|\xi|^{j-1}\(|\eta|+|(1-Q[\theta,\omega])\eta|\),|\partial_{\omega}\phi[\theta,\omega]|\>\nonumber\\&
 \lesssim \sum_{j=2}^6 \(\omega^{\frac{-46+6j}{6}}+\omega^{-\frac{4+j}{6}}   \)\|\eta\|_{l^2_{-1}}^j\lesssim \omega^{-1}\|\eta\|_{l^2_{-1}}^2.\label{eq:Lemdisc1}
\end{align}
Similarly, we have
\begin{align}\label{eq:Lemdisc2}
|\omega^{-1}\<f[\theta,\omega,\xi],\partial_{\theta}\phi[\theta,\omega]\>|\lesssim \omega^{-1}\|\eta\|_{l^2_{-1}}^2.
\end{align}
Finally,
\begin{align}\label{eq:Lemdisc3}
|\<|\xi|^6\xi,\partial_{\omega}\phi[\theta,\omega]\>|+|\omega^{-1}\<|\xi|^6\xi,\partial_{\theta}\phi[\theta,\omega]\>|\lesssim \omega^{-\frac{11}{6}}\|\eta\|_{l^2_{-1}}^7.
\end{align}
Therefore, from \eqref{eq:disceq} and \eqref{invAbound}--\eqref{eq:Lemdisc3}, we obtain \eqref{est:disc}.
\end{proof}

\begin{proof}[Proof of Proposition \ref{prop:main}]
By Lemma \ref{lem:Q}, \eqref{xi:est} and \eqref{eq:booteta}, we have
\begin{align*}
\|\xi\|_{X_T}\leq C(1+C_0\omega_1^{-\frac{5}{6}}+C_0(C_0\epsilon)^6)\epsilon,
\end{align*}
for some $C>0$
Thus, taking $C_0= 4C$ and $\omega_1$ sufficiently large and $\epsilon_0$ sufficiently small so that $C(C_0\omega_1^{-\frac{5}{6}}+C_0(C_0\epsilon_0)^6)\leq \frac{1}{4}C_0$, we have \eqref{bootass1}  with $C_0$ replaced by $C_0/2$.

Next, from \eqref{omegaequiv} and \eqref{est:disc}, we have
\begin{align*}
\|\dot{\theta}-\omega\|_{L^1\cap L^\infty(0,T)}+\|\omega^{-1}\dot{\omega}\|_{L^1\cap L^\infty}\leq \widetilde{C}(\omega_1^{-\frac{1}{3}}+C_0^6\epsilon^6)C_0\epsilon,
\end{align*}
for some $\widetilde{C}>0$.
Thus,  taking $\omega_1$ sufficiently large and $\epsilon_0$ sufficiently small so that 
$\widetilde{C}(\omega_1^{-\frac{1}{3}}+C_0^6\epsilon^6)\leq \frac{1}{2}$, we have \eqref{bootass2} with $C_0$ replaced by $C_0/2$.
\end{proof}

\begin{proof}[Proof of Theorem \ref{thm:main}]
By Proposition \ref{prop:main}, we have \eqref{bootass1} and \eqref{bootass2} with $T=\infty$.
In particular, this estimate implies the convergence of $\omega$ in \eqref{eq:main2} and the  bound on the first term in the inequality of  \eqref{eq:main2}.
Further, since $\|\xi\|_{\mathrm{Stz}}<\infty$, by standard argument we see that there exists $\xi_+$ s.t.\ $\|\xi(t)-e^{\im t\Delta}\xi_+\|_{l^2}\to 0$ as $t\to \infty$.
Therefore, we have \eqref{thm:main} and the bound on the 2nd term in the inequality  of \eqref{eq:main2}.
 \end{proof}
\appendix

\section{Proof of Proposition \ref{prop:nbs}}

Proposition \ref{prop:nbs} is a consequence of the following lemma.
\begin{lemma}\label{lem:anbs}
There exists $\epsilon>0$ s.t. there exists $\boldsymbol{\psi}(\cdot)=\{\psi_j(\cdot)\}_{j=0}^\infty \in C^\omega((-\epsilon,\epsilon),l^\infty(\N_0,\R))$ s.t.
$\phi_\omega=\omega^{1/6}\((1+\omega^{-1}\psi_0(\omega^{-1}))\delta_0 + \sum_{j\geq 1}\omega^{-j}\psi_j(\omega^{-1})\(\delta_j+\delta_{-j}\)\)$ solves \eqref{sp}.
Here, $\N_0=\N\cup\{0\}$.
\end{lemma}

\begin{remark}
By Lemma \ref{lem:anbs}, we see $\omega^{-1/6}\phi_{\omega}$ is an analytic function w.r.t.\ $\omega^{-1}$.
Further, we have the sharp exponential decay estimate
\begin{align*}
|\phi_\omega(x)|\sim \omega^{\frac{1}{6}-|x|}.
\end{align*}
\end{remark}

\begin{proof}
Substituting $\phi_{\omega}=\omega^{\frac{1}{6}}\(\delta_0+\varphi\)$ into \eqref{sp}, we have
\begin{align}\label{sp:a1}
\omega^{-1}\(-\Delta_{\mathrm{d}}\delta_0 -\Delta_{\mathrm{d}} \varphi\)+\varphi -\sum_{n=1}^6 {}_7C_n\delta_0 \varphi ^n-\varphi ^7=0.
\end{align}
We set $\varphi=\omega^{-1}\psi_0\delta_0+\sum_{j=1}\omega^{-j}\psi_j\(\delta_j+\delta_{-j}\)$, where $\psi_j\in \R$.
Then, setting $a=\omega^{-1}$, \eqref{sp:a1} is equivalent to
\begin{align*}
\mathbf{F}( \boldsymbol{\psi}, a )=0,
\end{align*}
where $\mathbf{F}=\{F_j\}_{j=0}^\infty$, with
\begin{align*}
F_j(\boldsymbol{\psi},a)=\begin{cases}
2-2a \psi_1 +2a\psi_0 +\psi_0 -\sum_{n=1}^7{}_7C_n a^{n-1} \psi_0 ^n& j=0,\\
-1-a\psi_0+2a\psi_1 -a^2\psi_2+\psi_1 -a^6\psi_1^7 & j=1,\\
-\psi_{j-1}+2a\psi_j - a^2 \psi_{j+1} + \psi_j - a^{6j}\psi_j & j\geq 2.
\end{cases}
\end{align*}
We have $\mathbf{F}\in C^{\omega} (l^\infty \times \R,l^\infty)$.
Further, setting $\mathbf{e}_j=\{\delta_{jk}\}_{k=0}^\infty\in l^\infty(\N_0)$ where $\delta_{jk}$ is the Kronecker's delta,  we have
\begin{align*}
\mathbf{F}\(\frac{1}{3}\mathbf{e}_0+\sum_{j\geq 1}\mathbf{e}_j,0\)=0,
\end{align*}
and
\begin{align*}
D_{\boldsymbol{\psi}}\mathbf{F}\(\frac{1}{3}\mathbf{e}_0+\sum_{j\geq 1}\mathbf{e}_j,0\)=-6(\mathbf{e}_0,\cdot)\mathbf{e}_0 + \sum_{j\geq 1}(\mathbf{e}_j,\cdot)\mathbf{e}_j - \sum_{j\geq 1}(\mathbf{e}_{j},\cdot)\mathbf{e}_{j+1},
\end{align*}
which is invertible (one can easily find the inverse writing down the $D_{\boldsymbol{\psi}}F$ in the matrix form).
Thus, by implicit function theorem we have the conclusion.
\end{proof}

\section*{Acknowledgments}
M. M.\ was supported by the JSPS KAKENHI Grant Number 19K03579, G19KK0066A and\\ JP17H02853.

Masaya Maeda

Department of Mathematics and Informatics,
Faculty of Science,
Chiba University,
Chiba 263-8522, Japan

{\it E-mail Address}: {\tt maeda@math.s.chiba-u.ac.jp}

\medskip

Masafumi Yoneda

Graduate School of Science and Engineering,
Chiba University,
Chiba 263-8522, Japan

{\it E-mail Address}: {\tt cama6585@chiba-u.jp}

\end{document}